\documentclass[12pt]{amsart}
\usepackage[pagewise]{lineno}
\usepackage{amsfonts,amsmath,amscd,amssymb,paralist,amsthm}
\usepackage{mathrsfs}
\usepackage{amsxtra}
\usepackage[mathscr]{eucal}
\usepackage{graphicx}
\usepackage{latexsym,graphicx,verbatim,nameref}
\usepackage[numbers,sort&compress]{natbib}
\usepackage[all,cmtip]{xy}
\usepackage{hyperref}
\usepackage{fancyhdr}
\usepackage{color}
\usepackage{tikz}
\usepackage{tikz-cd}
\usepackage{caption}
\usepackage{subfigure}
\usepackage[marginal]{footmisc}

\usetikzlibrary{arrows,scopes,matrix}
\numberwithin{equation}{section}

\theoremstyle{definition}

\makeatletter
\newtheorem*{rep@theorem}{\rep@title}
\newcommand{\newreptheorem}[2]{%
\newenvironment{rep#1}[1]{%
 \def\rep@title{#2 \ref{##1}}%
 \begin{rep@theorem}}%
 {\end{rep@theorem}}}
\makeatother

\newreptheorem{theorem}{Theorem}
\newreptheorem{corollary}{Corollary}
\newreptheorem{lemma}{Lemma}

\newtheorem{theorem}{Theorem}[section]
\newtheorem{corollary}[theorem]{Corollary}
\newtheorem{lemma}[theorem]{Lemma}
\newtheorem{proposition}[theorem]{Proposition}
\newtheorem*{theorem*}{Theorem}
\newtheorem*{proposition*}{Proposition}

\newtheorem{remark}[theorem]{Remark}

\newtheorem{question}[theorem]{Question}

\newtheorem*{claim*}{Claim}

\newtheorem*{conjecture*}{Conjecture}
\newtheorem*{observation*}{Observation}
\newtheorem*{question*}{Question}

\begin{document}

\title{Topological conjugation classes of tightly transitive subgroups of $\text{Homeo}_{+}(\mathbb{S}^1)$}

\date{\today}

\author{Hui Xu \& Enhui Shi}
\footnote{This work is supported by NSFC (No. 11771318, No. 11790274).}
\address[E. Shi]{School of Mathematical Sciences, Soochow University, Suzhou 215006, P. R. China}
\email{ehshi@suda.edu.cn}

\address[H. Xu]
{\noindent School of mathematical and sciences, Soochow University, Suzhou, 215006, P.R. China}
\email{20184007001@stu.suda.edu.cn}

\maketitle

\begin{abstract}
Let $\text{Homeo}_{+}(\mathbb{S}^1)$ denote the group of orientation preserving homeomorphisms of the circle $\mathbb{S}^1$. A subgroup $G$ of  $\text{Homeo}_{+}(\mathbb{S}^1)$  is  tightly transitive if it is topologically transitive and no subgroup $H$ of $G$ with $[G: H]=\infty$ has this property;
is almost minimal if it has at most countably many nontransitive points. In the paper, we determine all the topological conjugation classes of tightly transitive and almost minimal subgroups of $\text{Homeo}_{+}(\mathbb{S}^1)$ which are isomorphic to $\mathbb{Z}^n$ for any integer $n\geq 2$.
\end{abstract}
\section{Introduction and Preliminaries}

\subsection{Background}

Given a group $G$ and a topological space $X$, one basic question is to classify all the continuous actions of $G$ on $X$ up to topological conjugations.
Generally, in order to get satisfactory results, one should make some assumptions on the topology of $X$, the algebraic structure of $G$, and the dynamics of the action. Poincar\'e's classification theorem for minimal orientation preserving homeomorphisms on the circle $\mathbb S^1$ is the first celebrated result toward the answer to this question; the rotation numbers are complete invariants for such systems (see \cite{Po}). In \cite{Gh87}, Ghys classified all orientation preserving minimal group actions on the circle using bounded Euler class; this extended the previous theorem due to Poincar\'e (see also \cite{Gh01}).

 Minimality and topological transitivity can be viewed as two kinds of irreducibility for nonlinear group actions. Inspired by the previous works of Poincar\'e and Ghys, it is natural to study the classification of topologically transitive group actions on the circle. However, the phenomena of topological transitivity are much richer than that of minimality; we have to make stronger assumptions on the algebraic structure of acting groups
and on the dynamics of the action to get interesting results.  We should note that if we consider an orientation preserving minimal action on the circle by an amenable group, the action must factor through a commutative group action by rotations; this is an easy conclusion of the existence of invariant probability measures on the circle. Contrary to the case of minimal actions, many solvable groups  possess faithful topological transitive actions on the real line $\mathbb R$ (see \cite {SZ1}), which corresponds to the actions on $\mathbb S^1$ with a global fixed point.

\subsection{Notions and notations}

Denote by $\mathbb{S}^1$ the unit circle in the complex plane $\mathbb C$. In this paper, we want to study the classification of topologically transitive orientation preserving faithful group actions on $\mathbb S^1$. This is essentially the same as determining the conjugation classes of topologically transitive subgroups of $\text{Homeo}_{+}(\mathbb{S}^1)$. Before the statement of the main results, let us recall some notions.

Let $X$ be a topological space and $\text{Homeo}(X)$ be the homeomorphism group of $X$. Then for a subgroup $G$ of $\text{Homeo}(X)$, the pair $(X,G)$ is called a \textbf{\it{dynamical system}}. The \textbf{\it{orbit}} of $x\in X$ under $G$ is $Gx=\{gx: g\in G\}$. For a subset $A\subseteq X$, define $GA=\bigcup _{x\in A}Gx$. A subset $A\subseteq X$ is called $G$-\textbf{\it{invariant}} if $GA=A$. If $A$ is $G$-invariant, denote $G\mid_A$ the restriction of the action of $G$ to $A$. We call $x\in X$  a $n$-\textbf{\it{periodic  point}} of $G$ if the orbit $Gx$ consists of $n$ elements. For a homeomorphism $f\in G$, a point $x\in X$ is called a periodic point of $f$ if $x$ is a period point of the cyclic group $\langle f\rangle$ generated by $f$. Particularly if $Gx=\{x\}$, then we call $x$ a \textbf{\it{fixed point}} of $G$. Denote by $\text{P}(G)$ and $\text{Fix}(G)$ the sets of periodic points and fixed points of $G$ respectively; denote by $\text{P}(f)$ and $\text{Fix}(f)$ the sets of periodic points and fixed points of $f$ respectively.

For a dynamical system $(X,G)$, $G$ is said to be\textbf{\it{ topologically transitive}} if for any two nonempty open subsets $U$ and $V$ of $X$, there is some $g\in G$ such that $g(U)\cap V\neq \emptyset$. If there is some point $x\in X$ such that the orbit $Gx$ is dense in $X$ then $G$ is said to be\textbf{\it{ point transitive}} and such $x$ is called a \textbf{\it{transitive point}}. If $x$ is not a transitive point, then it is said to be a \textbf{\it{nontransitive point}}. It is well known that if $G$ is countable and $X$ is a Polish space without isolated points, the notions of topological transitivity and point transitivity are the same. $G$ is called {\it minimal} if every point of $X$ is a transitive point. A homeomorphism $f$ of $X$ is said to be {\it topologically transitive} (resp. {\it minimal}) if the cyclic group $\langle f\rangle$ is topologically transitive (resp. minimal).

Let $\mathbb{S}^1$ denote the circle. Denote by $\text{Homeo}_{+}(\mathbb{S}^1)$ the group of all orientation preserving homeomorphisms of $\mathbb{S}^1$. Two subgroups $G$ and $H$ of $\text{Homeo}_{+}(\mathbb{S}^1)$ are said to be \textbf{\it{topologically conjugate}} (or \textbf{\it{conjugate}} for short), if there is a  homeomorphism $\phi\in \text{Homeo}_{+}(\mathbb{S}^1)$ such that $\phi G\phi^{-1}=H$.  If $G$ is topologically transitive and no subgroup $F$ of $G$ with $[G:F]=\infty$ is topologically transitive, then  $G$ is said to be \textbf{\it{tightly transitive}}; $G$ is said to\textbf{\it{ almost minimal}} if there are at most countably many nontransitive points of $G$.

\subsection{Description of the main result}

In \cite{SZ2}, it determined all topological conjugation classes of tightly transitive almost minimal subgroups of $\text{Homeo}_{+}(\mathbb{R})$  which are isomorphic to $\mathbb{Z}^n$ for any integer $n\geq 2$. In this paper, we extend this result to group actions on the circle $\mathbb{S}^1$; that is, we determine all topological conjugation classes of  tightly transitive and almost minimal subgroups of $\text{Homeo}_{+}(\mathbb{S}^1)$ which are isomorphic to $\mathbb{Z}^n$ for any integer $n\geq 2$. Roughly speaking, all the conjugation classes are parameterized by a combination of orbits of irrational numbers under the action of $GL(2, \mathbb Z)$ by M\"obius transformations and orbits of $\mathbb Z^n$ under some specified affine actions (see Theorem 7.1). In fact, the Poincar\'e's classification theorem indicates that, for minimal subgroups of $\text{Homeo}_{+}(\mathbb{S}^1)$ which is isomorphic to $\mathbb Z$, all conjugation classes are parameterized by the orbits of irrationals under the $\mathbb Z$ action on $\mathbb R$ generated by the unit translation. Then we compare these two classification theorems in the following tublar. (``TT" and ``AM" denote the properties of tight transitivity and almost minimality respectively; $\mathcal{O}(...)$ denotes the orbits of ... ).

\vspace{5mm}

\begin{tabular}{|c|c|c|}\hline
  & Poincar\'e's classification & The present classification\\ \hline
Spaces & $\mathbb S^1$ & $\mathbb S^1$\\ \hline
Groups & $\mathbb Z$ & $\mathbb Z^n\ (n\geq 2)$\\ \hline
Dynamics & Minimality & TT \& AM\\ \hline
Invariants & $\mathcal{O}$(integer translations) &  $\mathcal{O}$(M\"obius \& affine actions) \\ \hline

\end{tabular}

\vspace{5mm}

Here we should remark that the subgroups of $\text{Homeo}_{+}(\mathbb{S}^1)$ constructed in the paper do not occur in  $\text{Diff}^{1+\varepsilon}(\mathbb{S}^1)$ for sufficiently large $\varepsilon \in (0, 1)$ (see e.g. \cite{DKA, Na11}),
and we do not plan to discuss the smooth realization of these groups in the present paper. Also, it may be worthwhile
to compare the actions of $\mathbb Z^n$ with that of lattices in higher rank simple Lie groups (higher rank lattices) on the circle. People believe
that there are no interesting actions for such lattices (see e.g. \cite{Bu, Na10, Gh99, Wi}). Certainly, the following question is left:

\begin{question}
For each finitely generated torsion free nilpotent group $\Gamma$, determine the topological conjugation classes of tightly transitive and almost
minimal subgroups $G$ of $\text{Homeo}_{+}(\mathbb{S}^1)$ which is isomorphic to $\Gamma$.
\end{question}
We recommend the readers to consult \cite{CJN, FF, JNR} for the discussions about nilpotent group actions on one-manifolds.

The paper is organized as follows. In section 2, we give some auxiliary results  which will be used in the following sections. In section 3, we recall and prove some results around group actions on $\mathbb{R}$, which is the starting point for further considerations. In section 4,  we construct a class of tightly transitive and almost minimal subgroups $G_{\alpha,n,k,g,f}$ of $\text{Homeo}_{+}(\mathbb{S}^1)$, which are isomorphic to $\mathbb Z^n$ and parameterized by five indices $\alpha,n,k,g,f$.  In section 5, we show that every tightly transitive and almost minimal subgroup of $\text{Homeo}_{+}(\mathbb{S}^1)$ which is isomorphic to $\mathbb{Z}^n$ is topologically conjugate to some $G_{\alpha,n,k,g,f}$. In section 6, we determine all the topological conjugation classes of these $G_{\alpha,n,k,g,f}$. In the last section, we restate the classification theorem in terms of matrix, with respect to a fixed standard basis of $G_{\alpha,n}$.

\section{Auxiliary results}
The following is the well-known Poinca\'{r}e's classification theorem (see e.g. \cite[Chap. 11]{K}).
\begin{theorem}\label{poincare's classification}
Let $f:\mathbb{S}^1\rightarrow\mathbb{S}^1$ be an orientation preserving homeomorphism.
\begin{itemize}
  \item [(1)] If $f$ has a periodic point, then all periodic orbits have the same period.
  \item [(2)] If $f$ has no periodic point, then there is a continuous surjection $\phi:\mathbb{S}^1\rightarrow \mathbb{S}^1$ and a minimal rotation $T: \mathbb{S}^1\rightarrow \mathbb{S}^1$ with $\phi f= T\phi$. Moreover, the map $\phi$ has the property that for each $z\in\mathbb{S}^1$, $\phi^{-1}(z)$ is either a point or a closed sub-interval of $\mathbb{S}^1$.
\end{itemize}
\end{theorem}
We collect the following useful properties for the subgroup of $\text{Homeo}_{+}(\mathbb{S}^1)$ with periodic orbits, which can be seen in \cite{Na11} Exercise 2.1.2.
\begin{proposition}\label{homeo with periodic orbit}
Let $H$ be a subgroup of  $\text{Homeo}_{+}(\mathbb{S}^1)$. If $H$ has a periodic orbit, then
\begin{itemize}
  \item [(1)] the set $P(H)$ of periodic points is a compact subset of $\mathbb{S}^1$;
  \item [(2)] all periodic orbits have the same cardinality.
\end{itemize}
\end{proposition}
Let $f$ be a homeomorphism on a topological space $X$. Recall that a point $x$ in $X$ is called a {\it wandering point} of $f$, if there
exists an open neighborhood $U$ of $x$ such that the sets $f^n(U)$, $n\in\mathbb Z$, are pairwise disjoint. We use $W(f, X)$ to denote
the set of all wandering points. Then $W(f, X)$ is an $f$ invariant open set. The following lemma is direct.

\begin{lemma}\label{wandering lemma}
Let $f$ and $g$ be homeomorphisms on a topological space $X$ such that $fg=gf$. Then (1) $g(W(f, X))=W(f, X)$; (2) if $x\in X$ is an $n$
periodic point of $f$, then $g(x)$ is also an $n$ periodic point of $f$.
\end{lemma}

\begin{lemma}\label{commuting lemma}
Let $H$ be a subgroup of $\text{Homeo}_{+}(\mathbb{S}^1)$ and $f\in \text{Homeo}_{+}(\mathbb{S}^1)$. If $f$ commutes with each element of $H$, and $P(H)\neq \emptyset$ , $P(f)\neq \emptyset$, then $P(H)\cap P(f)\neq \emptyset$ and the group $\langle f, H\rangle$ has a periodic orbit.
\end{lemma}

\begin{proof}
If $P(H)\subseteq P(f)$, then the conclusion holds. So we may assume that there is some $x\in P(H)\setminus P(f)$. Take a maximal interval
$(a,b)\subseteq \mathbb{S}^1\setminus P(f)$ with $x\in (a,b)$. Since $P(f)=Fix(f^p)$ for some positive integer $p$, by Theorem \ref{poincare's classification}, we have $a,b\in Fix(f^p)$. So, either $\lim_{n\rightarrow \infty} f^{np}(x)=a$ or $\lim_{n\rightarrow \infty} f^{-np}(x)=a$. We may assume that $\lim_{n\rightarrow \infty} f^{np}(x)=a$. By Lemma \ref{wandering lemma} (2), we have $f^{np}(x)\in P(H)$ for all $n\in\mathbb{Z}$. By Proposition \ref{homeo with periodic orbit}, $P(H)$ is compact. Thus $a\in P(H)$. Hence $P(H)\cap P(f)\neq \emptyset$.\\
\indent Let $y\in P(H)\cap P(f)$. Then $Hy=\{h_0(y), h_1(y),\cdots, h_{n-1}(y)\}$, for some $h_0,\cdots,h_{n-1}\in H$, and $f^k(y)=y$, for some positive integer $k$. Since $f$ commutes with $H$, we have
\[
\langle f, H\rangle y=\{f^mh(y): m\in\mathbb{Z}, h\in H \}=\bigcup_{i=0}^{n-1}\bigcup_{j=0}^{k-1}\{f^jh_i (y)\},
\]
which is finite. Hence the group $\langle f, H\rangle$ has a periodic orbit.
\end{proof}

\begin{proposition}\label{existence of finte orbit}
Let $G$ be a subgroup of $\text{Homeo}_{+}(\mathbb{S}^1)$ which is isomorphic to $\mathbb{Z}^n$ with $n\geq 2$, tightly transitive and almost minimal. Then there is a finite $G$-orbit $\{x_1,\cdots,x_k\}$ with $k\geq 1$.
\end{proposition}
\begin{proof}
By Lemma \ref{commuting lemma}, it suffices to show that $P(g)\neq \emptyset$ for any $g\in G$. Otherwise, let $f\in G$ be such that $P(f)=\emptyset$. Then there is a continuous surjection $\phi:\mathbb{S}^1\rightarrow \mathbb{S}^1$ and a minimal rotation $T: \mathbb{S}^1\rightarrow \mathbb{S}^1$ with $\phi f= T\phi$ by Theorem \ref{poincare's classification}. If $\phi$ is a homeomorphism, then $f$ is minimal which contradicts to the tight transitivity of $G$. Thus $W(f, \mathbb S^1)\not=\emptyset$ and $\mathbb S^1\setminus W(f, \mathbb S^1)$ is homeomorphic to the Cantor set. Since $\mathbb S^1\setminus W(f, \mathbb S^1)$ is $G$-invariant by Lemma \ref{wandering lemma} (1), each point of which is nontransitive.  This contradicts the almost minimality of $G$.
\end{proof}

The following lemma is well known. We afford a proof here for convenience of the readers.

\begin{lemma}\label{rotation lemma}
Let $T$ be a minimal rotation of $\mathbb{S}^1$  and $f\in \text{Homeo}(\mathbb{S}^1)$. If $f$ commutes with $T$, then $f$ is a rotation of $\mathbb{S}^1$.
\end{lemma}

\begin{proof}
Let $T: \mathbb{S}^1\rightarrow  \mathbb{S}^1, x\mapsto xe^{{\rm i}2\pi \theta}$, where $\theta$ is irrational. Since  $fT=Tf$,  $fT^n(1)=T^nf(1)$ for every integer $n$; that is $f(e^{{\rm i}2\pi n\theta})=e^{{\rm i}2\pi n\theta}f(1)$ for each $n$. Since $\{e^{{\rm i}2\pi n\theta}| n\in \mathbb Z\}$ is dense in $\mathbb S^1$, we have $f(x)=f(1)x$ for any $x\in\mathbb{S}^1$, by the continuity of $f$.
\end{proof}

For $a\in\mathbb R$, denote by $L_a$  the translation by $a$ on $\mathbb{R}$, i.e., $L_a(x)=x+a$ for $x\in \mathbb{R}$. The following proposition
can be deduced from Lemma \ref{rotation lemma} directly by quotienting the orbits of $L_1$.

\begin{proposition}\label{translation proposition}
Let $\alpha$ be an irrational number and $f\in \text{Homeo}_+(\mathbb R)$. If $f$ commutes with $L_1$ and $L_\alpha$ simultaneously, then $f=L_\beta$ for
some $\beta\in\mathbb R$.
\end{proposition}

\begin{lemma}[\cite{SZ2},Lemma 2.2]\label{Z2 action}
Let $H$ be a topologically transitive subgroup of $\text{Homeo}_{+}(\mathbb{R})$ which is isomorphic to $\mathbb{Z}^2$. Then $H$ is minimal.
\end{lemma}

\begin{lemma}\label{existence of minimal interval}
Let $H$ be a topologically transitive subgroup of $\text{Homeo}_{+}(\mathbb{R})$ which is isomorphic to $\mathbb{Z}^n$. Then there exists a nonempty open interval $(a,b)$ such that the restriction to $(a,b)$ of $F=\{h\in H: h(a,b)=(a,b)\}$ is minimal and the set of nontransitive points of $H$ is $\mathbb{R}\setminus \left(\bigcup_{h\in H}h(a,b)\right)$, where $a$ may be $-\infty$ and $b$ may be $+\infty$.
\end{lemma}

\begin{proof}
We prove the lemma by induction on the rank of $H$. Firstly, we have $n\geq 2$, by the fact that $H$ is topologically transitive. By Lemma \ref{Z2 action}, if $n=2$, then take $(a,b)=(-\infty,+\infty)$. So we may assume that $n\geq3$ and the action of $H$ is not minimal. Suppose that the assertions hold for any topologically transitive subgroup of $\text{Homeo}_{+}(\mathbb{R})$ which is isomorphic to $\mathbb{Z}^m$, with $m<n$. Let $x_0\in\mathbb{R}$ such that $\overline{Hx_0}\neq \mathbb{R}$. Let $(a_1,b_1)$ be a connected component of $\mathbb{R}\setminus \overline{Hx_0}$ and let $F_1=\{h\in H: h(a_1,b_1)=(a_1,b_1)\}$. \\

If $F_1|_{(a_1,b_1)}$ is minimal, then $F_1$ and $(a_1,b_1)$ satisfy the requirements. Suppose that it is not minimal. For any $h\in H$, either $h(a_1,b_1)=(a_1,b_1)$ or $h(a_1,b_1)\cap (a_1,b_1)=\emptyset$. Thus the restriction of $F_1$ to $(a_1,b_1)$ is topologically transitive. Since $H$ is topologically transitive, there exists $f\in H$ such that $f(a_1,b_1)\cap (a_1,b_1)=\emptyset$. Furthermore, because $f$ preserves the orientation of $\mathbb{R}$, we have $f^{k}(a_1,b_1)\cap (a_1,b_1)=\emptyset$, for any $k\in\mathbb{Z}, k\neq 0$. Thus $[H:F_1]=\infty$. Take an orientation preserving homeomorphism $\varphi: (a_1,b_1)\rightarrow \mathbb{R}$. Then $\varphi F_1|_{(a_1,b_1)}\varphi^{-1}$ is a topologically transitive subgroup of $\text{Homeo}_{+}(\mathbb{R})$ which is isomorphic to $\mathbb{Z}^{l}$, for some $l<n$. By induction hypothesis, there exists a nonempty open interval $(a_2,b_2)$ such that the assertions in the lemma hold for $\varphi F_1|_{(a_1,b_1)}\varphi^{-1}$. Take $(a,b)=\varphi^{-1}(a_2,b_2)\subseteq (a_1,b_1)$. Then the restriction to $(a,b)$ of $F=\{h\in H: h(a,b)=(a,b)\}$ is minimal.\\

For any topologically transitive point $x\in\mathbb{R}$ of $H$, there exists $h\in H$ such that $h(x)\in (a,b)$. Thus $x\in h^{-1}(a,b)\subseteq \bigcup_{h\in H}h(a,b)$. Note that $\bigcup_{h\in H}h(a,b)$ is $H$-invariant and in which there are topologically transitive points, we have $\overline{\bigcup_{h\in H}h(a,b)}=\mathbb{R}$. Since $F|_{(a,b)}$ is minimal, $Hy$ is dense in $\bigcup_{h\in H}h(a,b)$, for any $y\in \bigcup_{h\in H}h(a,b)$. Hence $Hy$ is dense in $\mathbb{R}$. Therefore, $\bigcup_{h\in H}h(a,b)$ is the very set of transitive points. Consequently, the set of nontransitive points is $\mathbb{R}\setminus \left(\bigcup_{h\in H}h(a,b)\right)$.
\end{proof}

\begin{lemma}\label{almost minimality of subgroup}
Let $G$ be a topologically transitive and almost minimal subgroup of $\text{Homeo}_{+}(\mathbb{R})$ which is isomorphic to $\mathbb{Z}^n$. For any subgroup $H$ of $G$, if $H$ is topologically transitive, then it is also almost minimal.
\end{lemma}

\begin{proof}
Since $H$ is topologically transitive, by Lemma \ref{existence of minimal interval}, there exists an interval $(a,b)$ such that the restriction to $(a,b)$ of $F=\{h\in H: h(a,b)=(a,b)\}$ is minimal, and the set of  nontransitive points of $H$ is $K:=\mathbb{R}\setminus \left(\bigcup_{h\in H}h(a,b)\right)$. Since $G$ is commutative, $K$ is a $G$-invariant closed set. Thus $K$ is contained in the set of   nontransitive points of $G$ which is countable by the almost minimality of $G$. Hence $H$ is also almost minimal.
\end{proof}

\section{Construction and Properties of $G_{\alpha,n}$}

In \cite{SZ2}, Shi and Zhou classified all the tightly transitive and almost minimal subgroups of $\text{Homeo}_{+}(\mathbb{R})$, which are isomorphic to
$\mathbb Z^n$ for any integer $n\geq 2$. These results are the starting point of the proof of the main theorem in this paper.

We first review the main results in \cite{SZ2}.  Let $\alpha$ be an irrational number in $(0,1)$ and $n\geq 2$ be an integer. Let $a,b\in\mathbb{R}$.  Denote by $\langle L_a, L_b\rangle$ the subgroup of $\text{Homeo}_{+}(\mathbb{R})$ generated by  $L_a$ and $L_b$ .

We define $G_{\alpha,n}$ inductively. Let $G_{\alpha,2}=\langle L_1, L_\alpha\rangle$. Suppose that we have constructed $G_{\alpha,n}$ for $n\geq 2$. Then
we construct $G_{\alpha, n+1}$ as follows. Choose a homeomorphism $\hbar$ from $\mathbb{R}$ to $(0,1)$. For example we can take
\begin{displaymath}
\hbar(x)=\frac{1}{\pi}\left(\frac{\pi}{2}+\arctan x\right)~~~\text{for } ~x\in\mathbb{R}.
\end{displaymath}

Then $\hbar$ induce an automorphism of  $\text{Homeo}_{+}(\mathbb{R})$ defined by
\begin{equation}\label{definition of h action}
\hbar(\sigma)(x)=\left\{
\begin{array}{cl}
\hbar \sigma \hbar^{-1}(x-i)+i,& x\in (i,i+1)~~\text{and} ~i\in\mathbb{Z},\\
x,& x\in\mathbb{Z},
\end{array}
\right.
\end{equation}
for $\sigma\in \text{Homeo}_{+}(\mathbb{R})$. Here we use the same symbol $\hbar$ to represent the automorphism, which will not lead to confusion from the text. For $n\in\mathbb{N}^+$, denote
\begin{displaymath}
\hbar^{(n)}(\sigma):=\hbar(\hbar(\cdots(\hbar(\sigma))\cdots)).
\end{displaymath}
By the definition, we immediately have the following relation.

\begin{lemma}\label{h morphism}
For any $\sigma_1,\sigma_2\in \text{Homeo}_{+}(\mathbb{R})$,
\begin{displaymath}
\hbar(\sigma_1\sigma_2)=\hbar(\sigma_1)\hbar(\sigma_2),~\text{and } \hbar(\sigma_1)L_1=L_1\hbar(\sigma_1).
\end{displaymath}
Furthermore, for any $m,n,k\in\mathbb{N}$,
\[
\hbar^{(m)}(\sigma_1)\hbar^{(n)}(L_1^k)=\hbar^{(n)}(L_1^k)\hbar^{(m)}(\sigma_1).
\]
\end{lemma}
\begin{proof}
For $x\in\mathbb{Z}$, $\hbar(\sigma_1)\hbar(\sigma_2)(x)=\hbar(\sigma_1)(x)=x=\hbar(\sigma_1\sigma_2)(x)$ and
\[
\hbar(\sigma_1)L_1(x)=\hbar(\sigma_1)(x+1)=x+1=L_1(x)=L_1\hbar(\sigma_1)(x).
\]
For $x\in(i,i+1), i\in\mathbb{Z}$,
\begin{eqnarray*}
\hbar(\sigma_1)\hbar(\sigma_2)(x)&=&\hbar(\sigma_1)\big(\hbar\sigma_2\hbar^{-1}(x-i)+i\big)\\
&=&\hbar\sigma_1\hbar^{-1}\Big(\big(\hbar\sigma_2\hbar^{-1}(x-i)+i\big)-i\Big)+i\\
&=&\hbar\sigma_1\sigma_2\hbar^{-1}(x-i)+i\\
&=&\hbar(\sigma_1\sigma_2)(x),
\end{eqnarray*}
and
\begin{eqnarray*}
\hbar(\sigma_1)L_1(x)&=&\hbar(\sigma_1)(x+1)=\hbar\sigma_1\hbar^{-1}(x-i)+i+1\\
&=&\hbar(\sigma_1)(x)+1=L_1\hbar(\sigma_1)(x).
\end{eqnarray*}
For the last assertion, we may assume that $m\geq n$. Then
\begin{eqnarray*}
\hbar^{(m)}(\sigma_1)\hbar^{(n)}(L_1^k)
&=&\hbar^{(n)}\Big(\hbar^{(m-n)}(\sigma_1)L_1^k\Big)\\
&=&\hbar^{(n)}\Big(L_1^k\hbar^{(m-n)}(\sigma_1)\Big)\\
&=&\hbar^{(n)}(L_1^k)\hbar^{(m)}(\sigma_1).
\end{eqnarray*}
\end{proof}

\begin{figure}[htp]
\centering
\includegraphics[width=13cm]{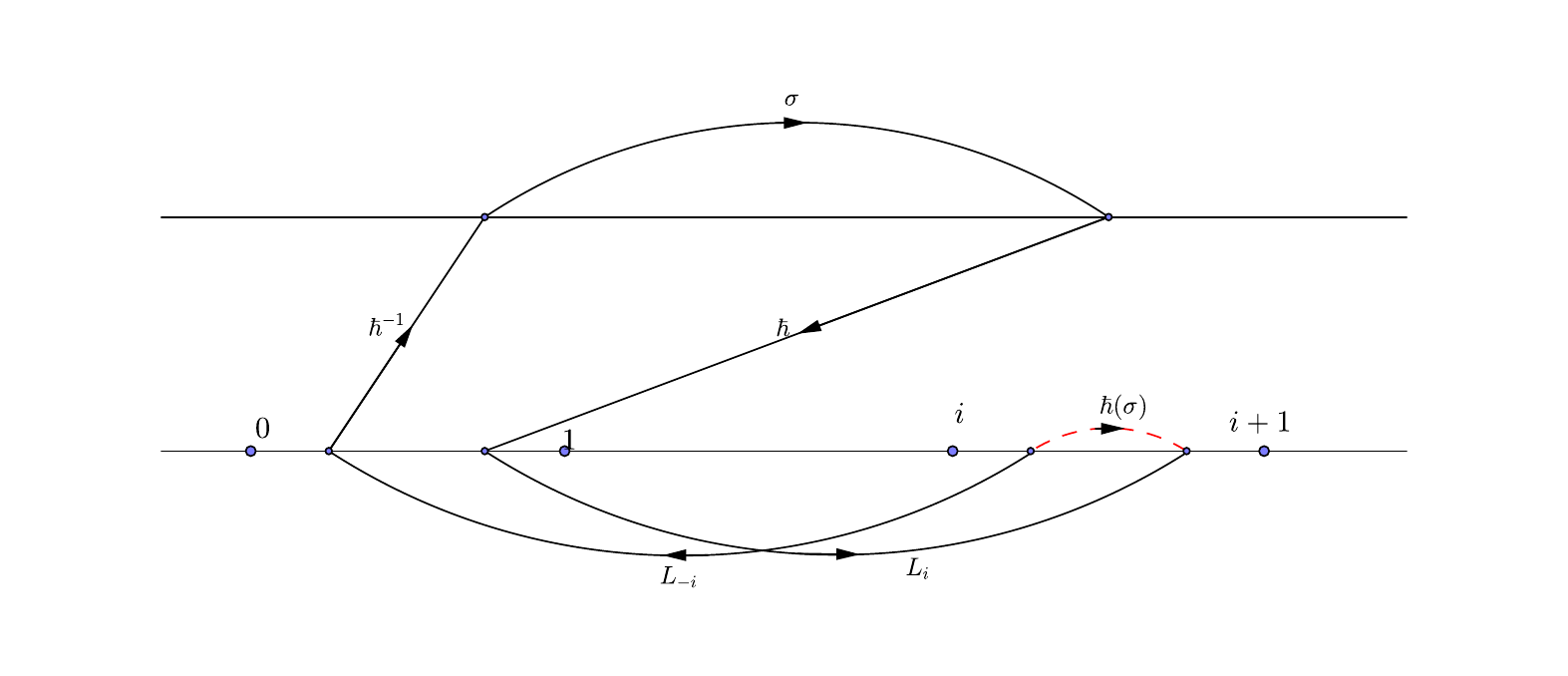}\\
\caption{Definition of $\hbar(\sigma)$}
\end{figure}
\noindent Let $G_{\alpha,n+1}$ be the group generated by $\{\hbar(\sigma): \sigma\in G_{\alpha,n}\}\cup\{L_1\}$. Then the constructed $G_{\alpha,n}$ has the following properties.

\begin{lemma}\label{properties of G(alpha,n)}
Let $\text{intr}G_{\alpha,n}$ denote the set of nontransitive points of $G_{\alpha,n}$. Then
\begin{itemize}
  \item [(1)] $G_{\alpha,n}$ is tightly transitive and is isomorphic to $\mathbb{Z}^n$.
  \item [(2)] $\text{intr} G_{\alpha,2}=\emptyset, \text{intr} G_{\alpha,3}=\mathbb{Z}$, and for $n\geq 4$,
  \begin{eqnarray*}
  \text{intr} G_{\alpha,n}&=&\left(\bigcup_{i\in \mathbb{Z}}\hbar(\text{intr} G_{\alpha,n-1})+i\right)\cup\mathbb{Z}\\
  &=& \mathbb{Z}\cup\bigcup_{i_1,\cdots,i_{n-2}\in\mathbb{Z}}\left\{\hbar\big(\hbar\big(\cdots (\hbar(i_1)+i_2)\cdots\big)+i_{n-3}\big)+i_{n-2}\right\}.
  \end{eqnarray*}
  \item [(3)] Suppose that $(a,b)$ is a connected component of $\mathbb{R}\setminus \text{intr}G_{\alpha,n}$ and $F=\{\sigma\in G_{\alpha,n}: \sigma((a,b))=(a,b)\}$. Then $F\mid_{(a,b)}$ is minimal and isomorphic to $\mathbb{Z}^2$. (Such an open interval $(a, b)$ is called a {\it minimal interval}.) Precisely,
      the minimal interval $(a,b)$ of $G_{\alpha,n}$ is of the following form:
\begin{enumerate}
  \item[a)]  if $n=2$, then $(a,b)=(-\infty,+\infty)$;
  \item[b)]  if $n=3$, then $(a,b)=(i,i+1)$, for some $i\in\mathbb{Z}$;
  \item[c)]  if $n\geq 4$, then $(a,b)=\hbar\big(\hbar\big(\cdots \hbar\big((i_1,i_1+1)+i_2)\big)\cdots\big)+i_{n-3}\big)+i_{n-2}$, for some $i_1,i_2,\cdots,i_{n-2}\in\mathbb{Z}$.
\end{enumerate}
\end{itemize}
\end{lemma}

Suppose that $\alpha$ and $\beta$ are irrationals in $(0,1)$. We say that $\alpha$ is \emph{equivalent} to $\beta$ if there exist $m_1,n_1,m_2,n_2\in \mathbb{Z}$ with $\left|m_1n_2-n_1m_2\right|=1$ such that $\beta=\frac{m_1+n_1\alpha}{m_2+n_2\alpha}$. The following theorem completes the classification of tightly transitive and almost minimal subgroups of $\text{Homeo}_{+}(\mathbb{R})$, which are isomorphic to $\mathbb Z^n$ with $n\geq 2$.

\begin{theorem}\label{classification of R} The following assertions hold:

\begin{itemize}
\item[(1)] For any $n\geq 2$ and  irrationals $\alpha,\beta\in(0,1)$, the subgroup $G_{\alpha,n}$ is conjugate to $G_{\beta,n}$ by an orientation preserving homeomorphism if and only if $\alpha$ is equivalent to $\beta$.
\item[(2)] Let $G$ be a tightly transitive and almost minimal subgroup of $\text{Homeo}_{+}(\mathbb{R})$ which is isomorphic to $\mathbb{Z}^n$ for some $n\geq 2$. Then $G$ is conjugate to $G_{\alpha,n}$ by an orientation preserving homeomorphism for some irrational $\alpha\in(0,1)$.
\end{itemize}

\end{theorem}

From the construction  of $G_{\alpha,n}$, we can define a basis $\{e_1, ..., e_n\}$ of $G_{\alpha,n}$ as a $\mathbb Z$ module.
For $G_{\alpha,2}$, we take $e_1=L_1$ and $e_2=L_\alpha$.
Generally, for $n\geq 3$, take
\begin{eqnarray*}
&&e_1=\hbar^{(n-2)}(L_1),e_2=\hbar^{(n-2)}(L_{\alpha}),\\&&e_3=\hbar^{(n-3)}(L_{1}), \cdots,\\&&e_{n-1}=\hbar^{(1)}(L_{1}), e_n=L_1.
\end{eqnarray*}
 The basis
$\{e_1, ..., e_n\}$ so defined is called the {\it standard basis} of $G_{\alpha,n}$.

Now we prove a technical lemma.

\begin{lemma}\label{substitution}
Suppose that  $\{ e_1, \cdots, e_n\}$ is the standard basis of $G_{\alpha,n}$, where $n\geq 2$ and $\alpha$ is an irrational number in $(0,1)$. If $f\in\text{Homeo}_{+}(\mathbb{S}^1)$ commutes with every element of $G_{\alpha,n}$ and $f\neq \text{id}$, then there exists some $i\in \{1, 2\}$ such that the group $\langle f, e_i, e_3, ..., e_n\rangle$ is also tightly transitive, almost minimal and isomorphic to $\mathbb{Z}^n$.
\end{lemma}
\begin{proof}
For $n=2$, we know that $f=L_{\beta}$ for some $\beta\in \mathbb{R}\setminus\{0\}$ by Proposition \ref{translation proposition}. If $\beta\in \mathbb{Q}$, then $\langle L_\beta, L_{\alpha}\rangle$ is tightly transitive and minimal. If $\beta$ is irrational, then $\langle L_1, L_\beta\rangle$ is tightly transitive and minimal. Thus the conclusion holds for $G_{\alpha, 2}$.\\

For $n\geq 3$, by Lemma \ref{properties of G(alpha,n)}, there is a minimal interval $(a,b)$.  By Lemma \ref{properties of G(alpha,n)} (3), we can take $(a,b)=\hbar\big(\cdots\hbar\big((0,1)\big)\cdots\big)$, where the number of the iterations is $n-3$. Set $F=\{\sigma\in G_{\alpha,n}: \sigma((a,b))=(a,b)\}$. Then, by the definition of standard basis, $F=\langle e_1, e_2\rangle$.  Since $f$  commutes with every element of $G_{\alpha,n}$, $f(a,b)$ is still a minimal interval of $G_{\alpha,n}$. By the structure of the minimal interval, there exist $k_3,\cdots,k_n\in\mathbb{Z}$ such that
\[
f(a,b)=\hbar\big(\hbar\big(\cdots \hbar\big((k_3,k_3+1)+k_4)\big)\cdots\big)+k_{n-1}\big)+k_{n}.
\]
Then $fe_3^{-k_3}...e_n^{-k_n}(a,b)=(a,b)$.\\

Let $g'=fe_3^{-k_3}...e_n^{-k_n}$. Define $g\in \text{Homeo}_{+}(\mathbb{S}^1)$ by
$g(x)=\hbar^{-(n-2)} g' \hbar^{n-2}(x)$. Then $g'$ commutes with $L_1$ and $L_\alpha$. By Proposition \ref{translation proposition}, there exists some $\theta\in \mathbb{R}\setminus\{0\}$ such that $g=L_{\theta}$.\\

We claim that $g'=\hbar^{(n-2)}(g)$. By the choice of $(a,b)=\hbar\big(\cdots\hbar\big((0,1)\big)\cdots\big)$, we have, for $x\in(a,b)$,
\[
\hbar^{(n-2)}(g)(x)=\hbar^{(n-3)}(g)(x)=\cdots=g'(x).
\]
For $x\in\{a,b\}$, it is obvious that $\hbar^{(n-2)}(g)(x)=g'(x)$. Now for any $x\in \mathbb{R}\setminus\mathbb{Z}$, there exist $j_3,\cdots,j_{n}\in\mathbb{Z}$ such that $e_3^{j_3}\cdots e_{n}^{j_n}(x)\in[a,b]$. Set $q=e_3^{j_3}\cdots e_{n}^{j_n}$. Thus
\[
g'(x)=q^{-1}g'q(x)=q^{-1}(\hbar^{(n-2)}(g))q(x)=\hbar^{(n-2)}(g)(x).
\]
The last equality follows by Lemma \ref{h morphism}.
As for $x\in\mathbb{Z}$, it is obvious that $g'(x)=\hbar^{(n-2)}(g)(x)=x$. Thus the claim is followed.\\

Now it is clear that $\langle f, e_i, e_3, ..., e_n\rangle=\langle g', e_i, e_3, ..., e_n\rangle$, for $i\in\{1,2\}$.\\

If $\theta$ is irrational, then $\langle f, e_1, e_3,...,e_n\rangle= \langle \hbar^{(n-2)}(L_{\theta}), e_1, e_3,...,e_n\rangle=G_{\theta,n}$  satisfies the requirements. If $\theta$ is rational, then $\langle f, e_2, e_3,...,e_n\rangle= \langle \hbar^{(n-2)}(L_{\theta}), e_2, e_3,...,e_n\rangle$ also satisfies the requirements. Indeed, in this case, the set of nontransitive points is
 \[
\mathbb{Z}\cup\bigcup_{i_1,\cdots,i_{n-2}\in\mathbb{Z}}\left\{\hbar\big(\hbar\big(\cdots (\hbar(i_1)+i_2)\cdots\big)+i_{n-3}\big)+i_{n-2}\right\},
 \]
 which is countable. Hence $\langle f, e_2, e_3,...,e_n\rangle$ is almost minimal. Let $H=\langle f, e_2, e_3,...,e_n\rangle$ and suppose that $F$ is topologically transitive subgroup of $H$. Note that $(a,b):=(\hbar^{n}(0), \hbar^{n}(1))$ is a minimal interval of $H$. Let $E=\{h\in H: h(a,b)=(a,b)\}$. Then $E=\langle f, e_2 \rangle$. By the topological transitivity of $F$, $(F\cap E)|(a,b)$ is also topologically transitive, whence $(F\cap E)|(a,b)\cong \mathbb{Z}^2$. For $i=3,\cdots,n$, $e_i$ is the unique element of $H$ that maps $(a,b)$ to another minimal interval $e_i(a,b)$. Thus $\{e_3,\cdots,e_n\}\subseteq F$. Hence $F\cong \mathbb{Z}^n$, which means that $F$ is a subgroup of $H$ of finite index. Therefore, $H$ is tightly transitive. This completes the proof.
\end{proof}

\section{Construction and Properties of $G_{\alpha,n,k,g,f}$}

Let integers  $n\geq 2$ and $k\geq 1$. Let $\alpha$ be an irrational number in $(0,1)$. Let $G_{\alpha,n}^s=\{g^s: g\in G_{\alpha,n}\}$, for $s\in\mathbb{N}^*$. Suppose  $g\in G_{\alpha, n}\setminus \bigcup_{gcd(k,s)\neq 1}G_{\alpha,n}^s$.  Put $x_j=e^{{\rm\bf i}2\pi j/k}$ for $j=1,...,k$. Denote by $(x_i, x_{i+1})~(\text{resp. }[x_i,x_{i+1}])$ the open (resp. closed) interval from $x_i$ to $x_{i+1}$ anticlockwise. Fix an orientation preserving homeomorphism $\phi$ from $\mathbb{R}$ to $(x_1, x_2)$. Then $\phi$ define a homomorphism:
\[
\widetilde{\qquad } :\text{Homeo}_{+}(\mathbb{R})\longrightarrow \text{Homeo}_{+}((x_1,x_2)),~~\sigma\mapsto \tilde{\sigma}=\phi\sigma\phi^{-1}.
\]

Let $f\in \text{Homeo}_{+}(\mathbb{S}^1)$ be such that
\begin{itemize}
  \item $f(x_i)=x_{i+1 }$,  for $i=1,...,k$.
  \item $f^{k}\mid_{(x_{1},x_2)}=\tilde{g}: (x_{1},x_2)\rightarrow (x_1,x_2)$.
\end{itemize}
In the above definition, we take $x_{k+1}=x_1$. In the reminder of the paper we take this convention as well. We denote the collection of such $f\in \text{Homeo}_{+}(\mathbb{S}^1)$ by $\text{Homeo}_{+}(\mathbb{S}^1)_{k,g}$.\\

Now we define a homomorphism:
\[
\overset{\overline{\quad}}{\quad} \text{~or~}  \overset{\overline{\quad}f} {\quad}: \text{Homeo}_{+}(\mathbb{R})\longrightarrow \text{Homeo}_{+}(\mathbb{S}^1),~~\sigma\mapsto \overline{\sigma}^f,
\]
where $\overline{\sigma}^f$ is an orientation preserving homeomorphism of $\mathbb{S}^1$ defined by
\begin{equation}\label{extension}
\overline{\sigma}^f(x)=\left\{
\begin{array}{cll}
f^{i-1}\tilde{\sigma}f^{-(i-1)}(x),& x\in (x_i,x_{i+1}),& i=1,..., k,\\
x_{i},& x=x_i, & i=1,...,k.
\end{array}
\right.
\end{equation}
We denote $\overline{\sigma}^f$ by  $\overline{\sigma}$ for short when it is clear that $\sigma$ is extended by $f$.\\
\begin{remark}
Note that $f$ does not commute with $\overline{\sigma}^f$ in general. $f$ commutes with $\overline{\sigma}^f$ if and only if $\sigma$ commutes with $g$. In particular,  $\overline{\sigma}^f$ commutes with $f$, for any $\sigma \in G_{\alpha,n}$.
\end{remark}
 Now we define $G_{\alpha,n,k,g,f}$ to be the subgroup of $\text{Homeo}_{+}(\mathbb{S}^1)$ generated by $\{\overline{\sigma}:\sigma\in G_{\alpha, n}\}\cup\{f\} $.\\
\begin{figure}[htp]
\centering
\includegraphics[width=8cm]{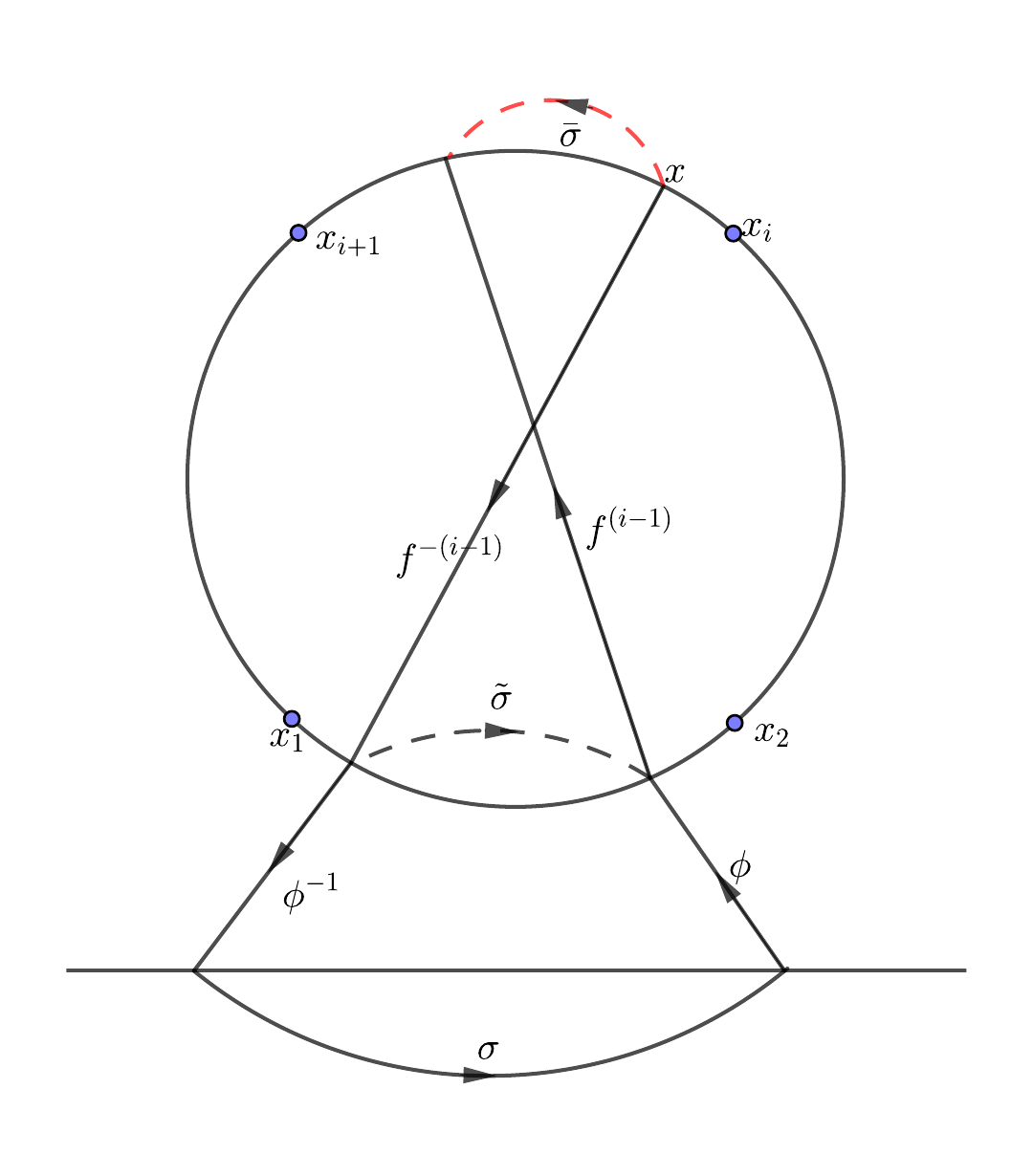}\\
\caption{Definition of $\overline{\sigma}^{f}$}
\end{figure}

By the above construction, we immediately have
\begin{lemma} The following assertions hold.

\begin{itemize}
  \item [(1)] Let $H=\{ \varphi\in G_{\alpha,n,k,g,f}: \varphi((x_1, x_2))=(x_1, x_2)\}$, then $H=\overline{G_{\alpha,n}} =\{\overline{\sigma}:\sigma\in G_{\alpha,n}\}$, $H|_{(x_1,x_2)}=\widetilde{G_{\alpha,n}}=\phi G_{\alpha,n}\phi^{-1}$ and $G_{\alpha,n,k,g,f}/H\cong \mathbb{Z}/k\mathbb{Z}$. Moreover,
  \begin{displaymath}
  G_{\alpha,n,k,g,f} =H\cup f H\cdots \cup f^{k-1}H.
  \end{displaymath}
  \item [(2)] $G_{\alpha,n,k,g,f}$ is a tightly transitive and  almost minimal subgroup of $\text{Homeo}_{+}(\mathbb{S}^1)$.
  \item [(3)] $G_{\alpha,n,k,g,f}$ is isomorphic to $\mathbb{Z}^n$.
\end{itemize}
\end{lemma}
\begin{proof}
(1) is direct from the construction of $G_{\alpha,n,k,g,f}$. As for (2), it is clear that $G_{\alpha,n,k,g,f}$ is topologically transitive and  almost minimal, since $G_{\alpha,n}$ is tightly transitive and almost minimal. For any topologically transitive subgroup $F$ of $G_{\alpha,n,k,g,f}$, $(F\cap H)|_{(x_1,x_2)}$ is topologically transitive. By (1), $H|_{(x_1,x_2)}$ is tightly transitive. Thus $[H|_{(x_1,x_2)}: (F\cap H)|_{(x_1,x_2)}]<\infty$. Then $[G_{\alpha,n,k,g,f}: F]<\infty$, since $[G_{\alpha,n,k,g,f}:H]=k$. Therefore, $G_{\alpha,n,k,g,f}$ is tightly transitive.\\

It remains to show (3). Note that the possible torsion elements of $G_{\alpha,n,k,g,f}$ are of the form $f^j\bar{h}$ for some $h\in G_{\alpha,n}$ and $j=1,\cdots,k-1$. If it is a torsion element, then there exists a positive integer $r$ such that $(f^j\bar{h})^{kr}=\text{id}$. Particularly,
\begin{displaymath}
(f^j\bar{h})^{kr}\mid_{(x_1,x_2)}=\widetilde{g^{jr} h^{kr}}=(\widetilde{g^jh^k})^r=\text{id}.
\end{displaymath}
Then $(g^jh^k)^r=\text{id}$. Since $G_{\alpha,n}$ is torsion-free, we have $g^{j}=h^{-k}\in G_{\alpha,n}^k$. Since the group $\langle g,h\rangle$ is free and abelian, it is cyclic. Thus there exists $w\in G_{\alpha,n} $ such that $g=w^s, h=w^{-t}$ for some positive integers $s,t$. Then $sj=tk$. Since $1\leq j\leq k-1$, we have $gcd(s,k)\neq 1$. Therefore, $g\in  G_{\alpha, n}\setminus \bigcup_{gcd(k,s)\neq 1}G_{\alpha,n}^s$, which contradicts the choice of $g$. \\

Now we know that  $G_{\alpha,n,k,g,f}$ is a finitely generated and torsion-free abelian group. This together with the facts that $G_{\alpha,n,k,g,f}/H\cong \mathbb{Z}/k\mathbb{Z}$ and $H\cong \mathbb{Z}^n$ imply that  $G_{\alpha,n,k,g,f}$ is isomorphic to $\mathbb{Z}^n$.
\end{proof}

\begin{remark}\label{torsion}
In the above proof, we know that $G_{\alpha,n,k,g,f}$ is torsion free for $g\in  G_{\alpha, n}\setminus \bigcup_{gcd(k,s)\neq 1}G_{\alpha,n}^s$. Conversely, if $g\in G_{\alpha,n}^s$ with $gcd (s,k)\neq 1$, then there exist torsion elements. Indeed, if $gcd(s,k)=k_1\neq 1$ (we write $k=k_1k_2$ and $s=k_1s_1$) and $g=w^s\in G_{\alpha,n}^s$, then $f^{j}\overline{h^{-1}}$ with $j=k_2s_2$ and $h=w^{s_1s_2}$ is a torsion element for any integer $s_2$.
\end{remark}

\section{Tightly transitive subgroups of $\text{Homeo}_{+}(\mathbb{S}^1)$}

Suppose that $G$ is a tightly transitive and almost minimal subgroup of $\text{Homeo}_{+}(\mathbb{S}^1)$ which is isomorphic to $\mathbb{Z}^n$ with $n\geq 2$.
 It follows from Proposition \ref{existence of finte orbit} that $G$ has a finite orbit $\{x_1,\cdots,x_k\}$ for some $k\geq 1$. We assume that $x_1,\cdots,x_k$ are on the circle in the anticlockwise ordering.

\begin{proposition}\label{tightly transitive of stable subgroup}
 Let $H=\{g\in G:  g(x_i)=x_i,  1\leq i\leq k\}$.
Then $G/H\cong \mathbb{Z}/k\mathbb{Z}$. Moreover, the restriction of $H$ to $(x_1,x_2)$ is tightly transitive and almost minimal.
\end{proposition}

\begin{proof}
Take an $f\in G$ such that $f(x_1)=x_2$. Since $f$ is orientation preserving, we have $f(x_i)=x_{i+1}$ for each $i$. Thus $f^k(x_i)=x_i$ and $f^k\in H$.

For any $g\in G$, suppose that $g(x_1)=x_j$ for some $1\leq j\leq k$. Then $f^{-(j-1)}g(x_1)=x_1$ and so $f^{-(j-1)}g(x_i)=x_i$ for each $i$. Thus  $f^{-(j-1)}g\in H$ , that is  $g\in f^{(j-1)}H$. Therefore,
\begin{displaymath}
G/H\cong \mathbb{Z}/k\mathbb{Z}.
\end{displaymath}
It is clear that the restriction of $H$ to $(x_1,x_2)$ must be topologically transitive and almost minimal. It remains to show it is tightly transitive.

If the restriction of $H$ to $(x_1,x_2)$ is not tightly transitive, then there is a subgroup  $F$ of $H $ such that $F\mid_{(x_1,x_2)}$ is topologically transitive and $[H\mid_{(x_1,x_2)}: F\mid_{(x_1,x_2)}]=\infty$. We may as well assume that $F\mid_{(x_1,x_2)}$ is tightly transitive. By Lemma \ref{almost minimality of subgroup}, $F$ is almost minimal. Then, by Theorem \ref{classification of R} (2),  $F\mid_{(x_1,x_2)}$ is conjugate to $G_{\alpha, m}$ for some irrational $\alpha$ and $m<n$.

There are two cases:

\indent {\bf Case 1.} $f^k\mid_{(x_1,x_2)}\in F\mid_{(x_1,x_2)}$.  Then $\tilde{F}:=\langle F, f\rangle$ is a topologically transitive subgroup of  $\text{Homeo}_{+}(\mathbb{S}^1)$ and $[\tilde{F}: F]=k$. Hence $[G: \tilde{F}]=\infty$, since $[G: H]=k$ and $[H: F]=[H\mid_{(x_1,x_2)}: F\mid_{(x_1,x_2)}]=\infty$. We get a contradiction to the tight transitivity of $G$.

\indent {\bf Case 2.} $f^k\mid_{(x_1,x_2)}\notin F\mid_{(x_1,x_2)}$.  Then, by Lemma \ref{substitution}, there is a  subgroup  $F'$ of  $H$ such that
 $f^k\in F'$ and the restriction $F'\mid_{(x_1,x_2)}$ is tightly transitive, almost minimal and $F'\cong \mathbb{Z}^m$. Similar to Case 1, we get a contradiction again.
\end{proof}

By Proposition \ref{tightly transitive of stable subgroup} and Theorem \ref{classification of R}, we see that no point in $\mathbb S^1\setminus\{x_1,...,x_k\}$ has a finite $G$-orbit. So, we have
\begin{corollary}\label{uniqueness}
 $\{x_1,...,x_n\}$ is the unique finite $G$-orbit.
\end{corollary}

\begin{theorem}\label{all systems}
Let  $G$ be a tightly transitive and almost minimal subgroup of $\text{Homeo}_{+}(\mathbb{S}^1) $, which is isomorphic to $\mathbb{Z}^n$ for some $n\geq 2$. Then $G$ is topologically conjugate to some $G_{\alpha,n,k,g,f}$.
\end{theorem}

\begin{proof}
By Propostion \ref{existence of finte orbit}, there exists a finite $G$-orbit  $x_1,\cdots,x_k$ which lie on $\mathbb S^1$ in the anticlockwise ordering.
WLOG, we may assume  $x_j=e^{{\rm\bf i}2\pi j/k}$ for $j=1,...,k$ as in Section 4, otherwise we need only replace $G$ by some $G'$ conjugating to it.
 Let
\begin{displaymath}
H=\{g\in G:  g(x_i)=x_i,  1\leq i\leq k\}.
\end{displaymath}
Then $H\mid_{(x_1,x_2)}$ is tightly transitive, almost minimal and isomorphic to $\mathbb{Z}^n$ by Propostion \ref{tightly transitive of stable subgroup}. Therefore, by Theorem \ref{classification of R}, there exists an irrational $\alpha\in (0,1)$ such that $H\mid_{(x_1,x_2)}$ is conjugate to $\widetilde{G_{\alpha,n}}$. Precisely, let $\phi\in \text{Homeo}_{+}(\mathbb{R},(x_1,x_2))$ be as in the first paragraph of Section 4. Then there exists a $\psi\in \text{Homeo}_{+}(\mathbb{R})$ such that
\begin{displaymath}
 \psi \phi^{-1} H\mid_{(x_1,x_2)}\phi\psi^{-1}=G_{\alpha,n}.
\end{displaymath}
Let $f\in G$ be such that $f(x_1)=x_2$. Then $f^{k}\mid_{(x_1,x_2)}\in H\mid_{(x_1,x_2)}$. Let $g=\psi\phi^{-1} f^k\mid_{(x_1,x_2)}\phi\psi^{-1}\in G_{\alpha,n}$. By Remark \ref{torsion}, we have $g\notin \bigcup_{(s,k)\neq 1}G_{\alpha,n}^s$, since $G$ is torsion-free. Next we show that
\begin{displaymath}
\overline{\psi} G \overline{\psi}^{-1}=G_{\alpha,n,k,g,f}.
\end{displaymath}
Note that
\begin{equation*}
\overline{\psi} (x)=\left\{
\begin{array}{cl}
f^{(i-1)}(\phi\psi\phi^{-1})f^{-(i-1)}(x),& x\in (x_i,x_{i+1}),\\
x_{i},& x=x_i,
\end{array}
\right.
\end{equation*}
and $G=H\cup f H\cup\cdots\cup f^{k-1} H$.   For $x\in (x_i,x_{i+1})$ and $h\in H$,
 \begin{eqnarray*}
 &&\overline{\psi}h \overline{\psi}^{-1}(x)\\
  &=& \left[f^{(i-1)}(\phi\psi\phi^{-1})f^{-(i-1)}] h [f^{(i-1)}(\phi\psi^{-1}\phi^{-1})f^{-(i-1)}\right](x)\\
 &=&f^{(i-1)}(\phi\psi\phi^{-1})h\mid_{(x_1,x_2)}(\phi\psi^{-1}\phi^{-1})f^{-(i-1)}(x)\\
 &=&\overline{\psi\phi^{-1} h\mid_{(x_1,x_2)}\phi\psi^{-1}}(x)
 \end{eqnarray*}
Since $\overline{\psi}h \overline{\psi}^{-1}(x_i)=x_i$, we conclude that $\overline{\psi}h \overline{\psi}^{-1}\in G_{\alpha,n,k,g,f}$. It is clear that $\overline{\psi}f\overline{\psi}^{-1}=f$. Thus $\overline{\psi} G \overline{\psi}^{-1}\subseteq G_{\alpha,n,k,g,f}$.  It is similar for the converse direction. Thus
\begin{displaymath}
\overline{\psi} G \overline{\psi}^{-1}=G_{\alpha,n,k,g,f},
\end{displaymath}
which means that $G$ is topologically conjugate to some $G_{\alpha,n,k,g,f}$.
\end{proof}

\section{Classification of $G_{\alpha,n,k,g,f}$}
Theorem \ref{all systems} indicates that,  in order to determine all the conjugation classes of the concerned systems,
we need only classify the groups $G_{\alpha,n,k,g,f}$ defined in Section 4.

\begin{lemma}\label{disregarding f}
Let $n,k\in\mathbb{Z}$ with $n\geq2$ and $k\geq 1$, $\alpha$ be an irrational in $(0,1)$ and $g\in G_{\alpha, n}$. Then, for any $f,f'\in \text{Homeo}_{+}(\mathbb{S}^1)_{k,g}$, $G_{\alpha,n,k,g,f}$  is topologically conjugate to $G_{\alpha,n,k,g,f'}$.
\end{lemma}
\begin{proof}
Define $\psi\in \text{Homeo}_{+}(\mathbb{S}^1)$ by
\begin{equation*}
\psi(x)=\left\{
\begin{array}{cll}
f^{(i-1)}f'^{-(i-1)}(x),& x\in (x_i,x_{i+1}),& i=1,...,k;\\
x,& x=x_i,& i=1,...,k.
\end{array}
\right.
\end{equation*}
For any $\sigma\in G_{\alpha,n}$, recall that  $\overline{\sigma}^f\in G_{\alpha,n,k,g,f}$ and $\overline{\sigma}^{f'}\in G_{\alpha,n,k,g,f'}$ are defined by (\ref{extension}).
So, for $x\in (x_i,x_{i+1})$,
\begin{eqnarray*}
\psi\overline{\sigma}^{f'}\psi^{-1}(x)&=& f^{(i-1)} f'^{-(i-1)} f'^{(i-1)}\tilde{\sigma}f'^{-(i-1)}f'^{(i-1)}f^{-(i-1)}(x)\\
&=& f^{(i-1)}\tilde{\sigma}f^{-(i-1)}(x)\\
&=& \overline{\sigma}^f(x).
\end{eqnarray*}
It is obvious that $ \psi\overline{\sigma}^{f'}\psi^{-1}(x_i)=\overline{\sigma}^{f'}(x_i)=x_i$. Hence $\psi\overline{\sigma}^{f'}\psi^{-1}=\overline{\sigma}^f$.

In addition, if $x\in (x_i,x_{i+1})$ with $1\leq i\leq k-1,$ then
$$\psi f'(x)=f^if'^{-i}f'(x)=f^if'^{-i+1}(x)=ff^{i-1}f'^{-(i-1)}(x)=f\psi(x);$$
if $x\in (x_{k},x_{1})$, then
$$\psi f'(x)=f'^kf'^{-(k-1)}(x)=\tilde{g}f'^{-(k-1)}(x)=f^kf'^{-(k-1)}(x)=f\psi(x).$$

Altogether, we have
\begin{displaymath}
\psi G_{\alpha,n,k,g,f'}\psi^{-1}=G_{\alpha,n,k,g,f}.
\end{displaymath}
That is to say that $G_{\alpha,n,k,g,f}$  is topologically conjugate to $G_{\alpha,n,k,g,f'}$.
\end{proof}

\begin{lemma}\label{conj with g}
If $G_{\alpha,n,k,g,f}$ is topologically conjugate to $G_{\alpha',n',k',g',f'}$, then $k=k', n=n'$ and $\alpha$ is equivalent to $\alpha'$.
\end{lemma}

\begin{proof}

$n=n'$ is clear; $k=k'$ follows from the fact that all finite orbits of a group of circle homeomorphisms have the same cardinality (Lemma \ref{homeo with periodic orbit}); $\alpha$ being equivalent to $\alpha'$ follows from Theorem \ref{classification of R} and
the definition of $G_{\alpha,n,k,g,f}$.
\end{proof}

Let $N_{\text{Homeo}_{+}(\mathbb{R})}(G_{\alpha,n })$ denote the normalizer of $G_{\alpha,n}$ in $\text{Homeo}_{+}(\mathbb{R}) $, i.e.,
$N_{\text{Homeo}_{+}(\mathbb{R})}(G_{\alpha,n })=\left\{\varphi\in \text{Homeo}_{+}(\mathbb{R}):~\varphi G_{\alpha,n}\varphi^{-1}=G_{\alpha,n}\right\}.$
 Thus we get an affine action on $G_{\alpha,n}$ by the semidirect $N_{\text{Homeo}_{+}(\mathbb{R})}(G_{\alpha,n })\ltimes G_{\alpha,n}^k$: $(\varphi, f). g:=\varphi g\varphi^{-1}f$, for any $(\varphi, f)\in N_{\text{Homeo}_{+}(\mathbb{R})}(G_{\alpha,n })\ltimes G_{\alpha,n}^k$ and $g\in G_{\alpha,n}$.

\begin{lemma}
 $G_{\alpha,n,k,g,f}$ is topologically conjugate to $G_{\alpha,n,k,g',f'}$ if and only if $g$ and $g'$ are in the same orbit of the affine action on $G_{\alpha,n}$ by $N_{\text{Homeo}_{+}(\mathbb{R})}(G_{\alpha,n })\ltimes G_{\alpha,n}^k$.
\end{lemma}

\begin{proof}
\emph{Sufficiency}. Suppose that $g$ and $g'$ are in the same orbit of the affine action on $G_{\alpha,n}$ by $N_{\text{Homeo}_{+}(\mathbb{R})}(G_{\alpha,n })\ltimes G_{\alpha,n}^k$. Then there exist some $\varphi\in N_{\text{Homeo}_{+}(\mathbb{R})}(G_{\alpha,n })$ and $h_0\in G_{\alpha,n}$ such that
\[
g'=\varphi g\varphi^{-1}h_0^k=\varphi g(\varphi^{-1} h_0^k\varphi)\varphi^{-1}.
\]
Since $\varphi\in N_{\text{Homeo}_{+}(\mathbb{R})}(G_{\alpha,n })$, $\varphi^{-1} h_0^k\varphi\in G_{\alpha,n }$. Set $h=\varphi^{-1} h_0\varphi$. Thus $g'=\varphi g h^k \varphi^{-1}$.\\

Let $\{x_1,...,x_k\}$ and $\phi: \mathbb{R} \rightarrow (x_1,x_2)$ be defined as in Section 4. Then $\phi\varphi\phi^{-1}\in\text{Homeo}_{+}((x_1,x_2))$.\\

 We show two special cases  firstly.\\

\noindent\textbf{Claim 1}. If $g'=gh^k$, then $G_{\alpha,n,k,g,f}=G_{\alpha, n,k,g', f''}$ with $f''=f\overline{h}^{f}$, where the definition of $\overline{h}^f$ can consult (\ref{extension}). Thus $G_{\alpha,n,k,g',f'}$ is conjugate to $G_{\alpha, n,k,g', f''}$ by Lemma \ref{disregarding f}. \\

Indeed, let $\sigma$ be in $G_{\alpha,n}$. For any $x\in (x_i,x_{i+1}), 1\leq i\leq k$,
\begin{eqnarray*}
\overline{\sigma}^{f''}(x)&=&f^{(i-1)}(\overline{h}^{f})^{(i-1)}\tilde{\sigma}(\overline{h}^{f})^{-(i-1)}f^{-(i-1)}(x)\\
&=& f^{(i-1)}\tilde{h}^{(i-1)}\tilde{\sigma}\tilde{h}^{-(i-1)}f^{-(i-1)}(x)\\
&=& f^{(i-1)}\tilde{\sigma}f^{-(i-1)}(x)\\
&=&\overline{\sigma}^f(x).
\end{eqnarray*}
It is obvious that $\overline{\sigma}^{f''}(x_i)=x_i=\overline{\sigma}^{f}(x_i) $. Thus $\overline{\sigma}^{f''}=\overline{\sigma}^{f}$.
Note that
\begin{displaymath}
G_{\alpha,n,k,g,f}=\langle\{\overline{\sigma}^f: \sigma\in G_{\alpha,n}\}\cup\{f\}\rangle,
\end{displaymath}
and
\begin{displaymath}
G_{\alpha,n,k,g',f''}=\langle\{\overline{\sigma}^{f''}: \sigma\in G_{\alpha,n}\}\cup\{f''\}\rangle.
\end{displaymath}
In addition, $f''=f\overline{h}^f\in G_{\alpha,n,k,g,f}$ and $f=f^{''}(\overline{h}^{f})^{-1}=f^{''}(\overline{h}^{f^{''}})^{-1}\in G_{\alpha,n,k,g',f''}$ . Therefore,
\begin{displaymath}
G_{\alpha,n,k,g,f}=G_{\alpha, n,k,gh^k, f''}.
\end{displaymath}

\noindent\textbf{Claim 2}. If $g'=\varphi g\varphi^{-1}$ with $\varphi\in N_{\text{Homeo}_{+}(\mathbb{R})}(G_{\alpha,n })$, then $G_{\alpha,n,k,g,f}$ is conjugate to $G_{\alpha, n,k,g', f'}$ by $\Phi:=\overline{\varphi^{-1}}^{f}$, where  $f'=\Phi f \Phi^{-1}\in\text{Homeo}_{+}(\mathbb{S}^1)$.\\

Indeed, let $\sigma\in G_{\alpha,n}$. For $x\in(x_i,x_{i+1})$ with $1\leq i\leq k$,
\begin{eqnarray*}
\overline{\sigma}^{f'}(x)&=& f'^{(i-1)} \tilde{\sigma} f'^{-(i-1)}(x)\\
&=&\Phi f^{(i-1)} \Phi^{-1}\tilde{\sigma}\Phi f^{-(i-1)} \Phi^{-1}(x)\\
&=&\Phi f^{(i-1)} \widetilde{\varphi \sigma\varphi^{-1}} f^{-(i-1)} \Phi^{-1}(x)\\
&=&\Phi  \overline{(\varphi \sigma\varphi^{-1})}^f  \Phi^{-1}(x).
\end{eqnarray*}
It is obvious that $\overline{\sigma}^{f'}(x_i) =\Phi  \overline{(\varphi \sigma\varphi^{-1})}^f  \Phi^{-1}(x_i)=x_i$ for any $1\leq i\leq k$. Hence
\begin{displaymath}
\overline{\sigma}^{f'}= \Phi  \overline{(\varphi \sigma\varphi^{-1})}^f  \Phi^{-1}.
\end{displaymath}
Since $\varphi\in N_{\text{Homeo}_{+}}(\mathbb{R})(G_{\alpha,n })$, we have $\overline{(\varphi \sigma\varphi^{-1})}^f\in G_{\alpha,n,k,g,f}$. Thus
\begin{displaymath}
G_{\alpha,n,k,g,f}=\langle\{\overline{(\varphi\sigma\varphi^{-1})}^f: \sigma\in G_{\alpha,n}\}\cup\{f\}\rangle,
\end{displaymath}
Hence $G_{\alpha,n,k,g',f'}=\Phi G_{\alpha,n,k,g,f}\Phi^{-1} $.\\
\begin{figure}[htbp]
\centering
\subfigure[Claim1]
{
\begin{minipage}{5cm}
\centering
\includegraphics[scale=0.5]{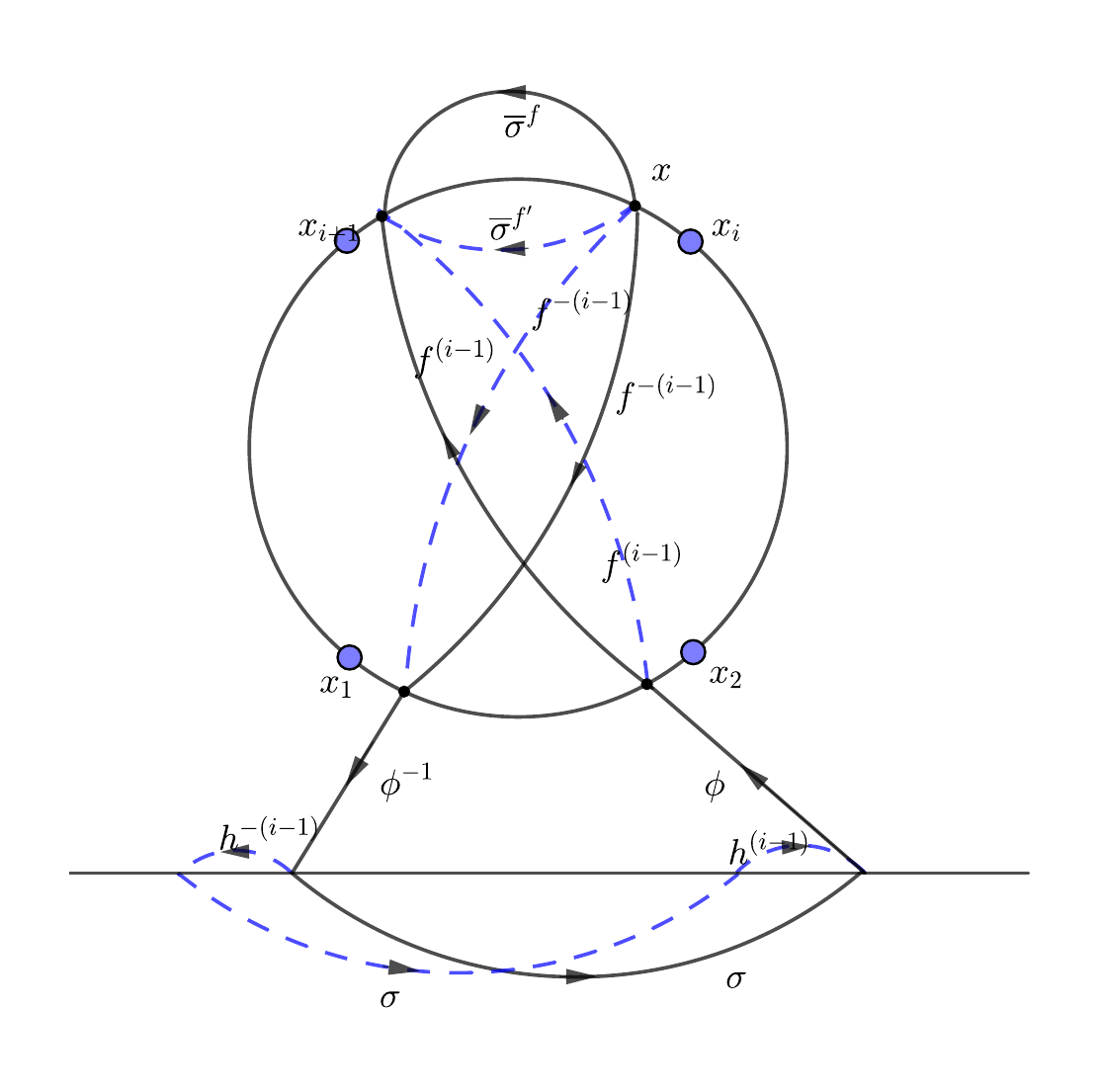}
\end{minipage}
}
\subfigure[Claim2]
{
\begin{minipage}{5cm}
\centering
\includegraphics[scale=0.5]{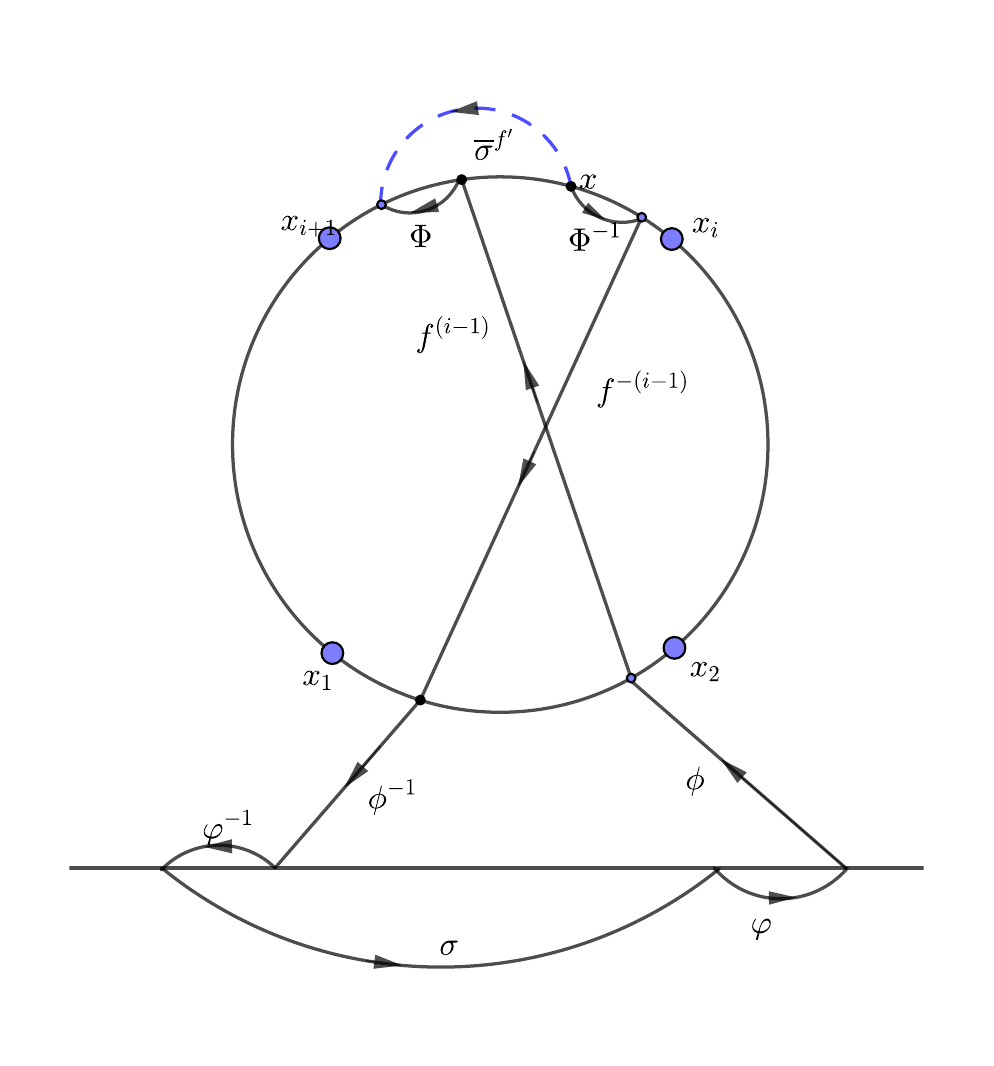}
\end{minipage}
}
\caption{Lemma 6.3}
\end{figure}

For general case, i.e. $g'=\varphi gh^k\varphi^{-1}$, we combine Claims 1 and 2 in order to obtain that $G_{\alpha,n,k,g,f}$ is conjugate to $G_{\alpha, n,k,g', f'}$. Precisely,  by Claim 1, we have $G_{\alpha,n, k,g,f}=G_{\alpha,n,k,gh^k, f\overline{h}^f}$. Then, by Claim 2, $G_{\alpha,n,k,gh^k, f\overline{h}^f}$ is conjugate to $G_{\alpha,n,k,\varphi gh^k\varphi^{-1}, (\overline{\varphi^{-1}}^{f\overline{h}^f}) f\overline{h}^f(\overline{\varphi}^{f\overline{h}^f})}$ by $\overline{\varphi^{-1}}^{f\overline{h}^f}$. By Lemma \ref{disregarding f}, $G_{\alpha,n,k,\varphi gh^k\varphi^{-1}, (\overline{\varphi^{-1}}^{f\overline{h}^f}) f\overline{h}^f(\overline{\varphi}^{f\overline{h}^f})}$ is conjugate to $G_{\alpha,n,k,g',f'}$. Hence $G_{\alpha,n,k,g,f}$ is conjugate to $G_{\alpha,n,k,g',f'}$. \\

\noindent \emph{Necessity}. Suppose that $G_{\alpha,n,k,g,f}$ is topologically conjugate to $G_{\alpha,n,k,g',f'}$. Then there exists $\psi\in\text{Homeo}_{+}(\mathbb{S}^1)$ such that
\begin{displaymath}
\psi G_{\alpha,n,k,g,f}\psi^{-1}=G_{\alpha,n,k,g',f'}.
\end{displaymath}

Note that the set $\{x_1,\cdots,x_k\}$ is $\psi$-invariant, since it represents the unique periodic orbit of both groups. Moreover, $\psi$ being orientation preserving, if $\psi(x_1)=x_1$, then $\psi(x_i)=x_i$ for all $i=1,\cdots,k$. We may assume that  $\psi(x_i)=x_i$ for any $1\leq i\leq k$ whence $\psi((x_i,x_{i+1}))=(x_i,x_{i+1})$. Otherwise, suppose that $\psi(x_1)=x_j$ for some integer $j$ with $1\leq j\leq k$. It is clear that $G_{\alpha,n,k,g,f}= f^{-(j-1)} G_{\alpha,n,k,g,f} f^{j-1}$. Thus
\begin{displaymath}
\psi f^{-(j-1)} G_{\alpha,n,k,g,f} f^{j-1}\psi^{-1}=G_{\alpha,n,k,g',f'}.
\end{displaymath}
Then $\psi f^{-(j-1)}$ satisfies the condition that $\psi f^{-(j-1)}(x_i)=x_i$ for any $1\leq i\leq k$.\\

Let $H=\{q\in G_{\alpha,n,k,g,f}: q(x_i)=x_i,i=1,\cdots,k\}=\left\{\overline{h}^f: h\in G_{\alpha,n}\right\}$. Then
\begin{displaymath}
fH=\{q\in G_{\alpha,n,k,g,f}:  q(x_i)=x_{i+1},i=1,\cdots, k\}.
\end{displaymath}
One has $\psi^{-1}f'\psi(x_i)=x_{i+1}$, for each $i=1,\cdots,k$. Thus $\psi^{-1}f'\psi\in f H$, that is
there exists an $h\in G_{\alpha,n}$ such that $\psi f \overline{h}^f\psi^{-1}=f'$. Thus
\begin{displaymath}
\psi\widetilde{gh^k}\psi^{-1}=\left(\psi (f\overline{h}^f)^k\psi^{-1}\right)\mid_{(x_1,x_2)}=f'^k\mid_{(x_1,x_2)}=\tilde{g'}.
\end{displaymath}
Let $\phi$ be the orientation preserving homeomorphism from $\mathbb{R}$  to $(x_1,x_2)$, fixed in Section 4. Thus
\begin{equation*}
\psi\phi g h^k \phi^{-1}\psi^{-1}=\phi g'\phi^{-1}.
\end{equation*}
Let $\varphi = \phi^{-1}\psi\phi$. Then $\varphi g h^k\varphi^{-1}=g'$.\\

It remains to show  $\varphi\in N_{\text{Homeo}_{+}(\mathbb{R})}(G_{\alpha,n })$.  It is obvious that $\varphi\in \text{Homeo}_{+}(\mathbb{R})$. For any $\sigma\in G_{\alpha,n}$,
\begin{displaymath}
\varphi \sigma\varphi^{-1} =\phi^{-1}\psi\tilde{\sigma}\psi^{-1}\phi.
\end{displaymath}
Note that $\psi|_{(x_1,x_2)}$ conjugates $G_{\alpha,n,k,g,f}|_{(x_1,x_2)}$ to $G_{\alpha,n,k,g',f'}|_{(x_1,x_2)}$, since they both coincide with $\widetilde{G_{\alpha,n}}$. Thus $\psi (\widetilde{G_{\alpha,n}})\psi^{-1}=\widetilde{G_{\alpha,n}}$, i.e., $\psi\phi G_{\alpha,n}\phi^{-1}\psi^{-1}=\phi G_{\alpha,n}\phi^{-1}$. Hence $\phi^{-1}\psi\phi G_{\alpha,n}\phi^{-1}\psi^{-1}\phi=G_{\alpha,n}$. Therefore, $\varphi=\phi^{-1}\psi\phi\in N_{\text{Homeo}_{+}(\mathbb{R})}(G_{\alpha,n })$.
\end{proof}

Define
$${\rm Conj}(G_{\alpha, n}, G_{\alpha', n})=\{\psi\in{\rm Homeo}_{+}(\mathbb R):G_{\alpha, n}=\psi G_{\alpha', n}\psi^{-1}\}.$$
If $\alpha$ and $\alpha'$ are equivalent, then ${\rm Conj}(G_{\alpha, n}, G_{\alpha', n})\not=\emptyset$ by Theorem \ref{classification of R};
and we fix a conjugation $\psi_{\alpha,\alpha'}\in{\rm Conj}(G_{\alpha, n}, G_{\alpha', n})$.\\

Analogous arguments to the Claim 2 in the proof of Lemma \ref{conj with g}, allow to construct a conjugation $\Psi$ of $G_{\alpha,n,k,g,f}$ to $G_{\alpha', n, k,g',f'}$, with $g'=\psi_{\alpha,\alpha'}g\psi_{\alpha,\alpha'}^{-1}$ and $f'=\Psi f\Psi^{-1}$. Finally, by the above lemmas, we obtain

\begin{theorem}\label{classification of S}
  $G_{\alpha,n,k,g,f}$ is topologically conjugate to $G_{\alpha',n',k',g',f'}$ if and only if
 \begin{itemize}
   \item $n=n'$ and $k=k'$;
   \item $\alpha$ is equivalent to $\alpha'$, i.e. there exist $m_1,n_1,m_2,n_2\in \mathbb{Z}$ with $\left|m_1n_2-n_1m_2\right|=1$ such that $\alpha'=\frac{m_1+n_1\alpha}{m_2+n_2\alpha}$;
   \item $g$ and $\psi_{\alpha,\alpha'} g'\psi_{\alpha,\alpha'}^{-1}$ are in the same orbit of the affine action on $G_{\alpha,n}$ by $N_{\text{Homeo}_{+}(\mathbb{R})}(G_{\alpha,n })\ltimes G_{\alpha,n}^k$, i.e. there exist some $\varphi\in N_{\text{Homeo}_{+}(\mathbb{R})}(G_{\alpha,n })$ and $h\in G_{\alpha,n}$ such that $\psi_{\alpha,\alpha'} g'\psi_{\alpha,\alpha'}^{-1}=\varphi gh^k\varphi^{-1}$.
 \end{itemize}
\end{theorem}

\section{Matrix Representation of $N_{\text{Homeo}_{+}(\mathbb{R})}(G_{\alpha,n })$}

In this section, we want to restate Theorem \ref{classification of S} in terms of matrix, with respect to the standard basis $\{e_1,...,e_n\}$ of $G_{\alpha, n}$ defined as in Section 3. This will make us easier to determine whether two systems $G_{\alpha,n,k,g,f}$, $G_{\alpha',n',k',g',f'}$  are conjugate.\\

From Theorem \ref{classification of R}-(1), we see that if $\alpha$ and $\beta$ are equivalent irrationals in $(0, 1)$, then
$G_{\alpha,n}$ and $G_{\beta, n}$ are conjugate. Now suppose that $\alpha=\frac{m_1+n_1\beta}{m_2+n_2\beta}$ for some integers $m_1, n_1, m_2, n_2$ with $|m_1n_2-n_1m_2|=1$. We will define a sequence of conjugations $\phi_n$ between $G_{\alpha,n}$ and $G_{\beta, n}$ for
every $n\geq 2$. When $n=2$, the conjugation $\phi_2$ between $G_{\alpha,2}$ and $G_{\beta, 2}$ can be
taken as a multiplication $M_u:\mathbb R\rightarrow \mathbb R, x\mapsto ux$, where $u=|m_2+n_2\beta|$ (see \cite[Lemma 3.3]{SZ2}). More precisely, we may assume that $m_2+n_2\beta>0$, otherwise we can replay $m_1,m_2,n_1,n_2$ by $-m_1,-m_2,-n_1,-n_2$ respectively. Then
\begin{equation}\label{2-conj}
M_u L_1M_{u}^{-1}=L_u=L_1^{m_2} L_{\beta}^{n_2}, \text{and }
M_u L_{\alpha}M_{u}^{-1}=L_{\alpha u}=L_1^{m_1} L_{\beta}^{n_1}.
\end{equation}
Since $|m_1n_2-n_1m_2|=1$ , we have
\begin{displaymath}
\mathbb{Z}^{2}\cong G_{\beta,2}=\langle L_1,L_{\beta}\rangle= \langle L_u,L_{\alpha u}\rangle.
\end{displaymath}
Therefore, $M_{u}G_{\alpha,2}M_{u}^{-1}=G_{\beta,2}$. Fix standard basis $\{L_1,L_{\alpha}\}$ , $\{L_1,L_{\beta}\}$ of $G_{\alpha,2}$ and $G_{\beta,2}$ respectively. Then, by \ref{2-conj}, the conjugation by $\phi_2=M_u$ can be represented by matrix
\[
\left(
\begin{array}{cc}
m_2& m_1\\
n_2& n_1
\end{array}
\right).
\]
More precisely, under the standard basis $\{L_1,L_{\alpha}\}$ of $\mathbb{Z}$-module $G_{\alpha,2}$, an element $g=L_1^{x_{\alpha}}L_{\alpha}^{y_{\alpha}}\in G_{\alpha,2}$ is represented by $\left(
\begin{array}{cc}
x_\alpha\\
y_{\alpha}
\end{array}
\right).$
Then the coordinate of $\phi_{2} g\phi^{-1}=L_1^{x_{\beta}}L_{\beta}^{y_{\beta}}$ under the basis $\{L_1,L_{\beta}\}$ is
\[
\left(
\begin{array}{cc}
x_\beta\\
y_\beta
\end{array}
\right)
=
\left(
\begin{array}{cc}
m_2& m_1\\
n_2& n_1
\end{array}
\right)
\left(
\begin{array}{cc}
x_\alpha\\
y_\alpha
\end{array}
\right).
\]
Now for $n\geq 3$, $\phi_n:=\hbar^{(n-2)}(\phi_2)$ is a conjugation between $G_{\alpha,n}$ and $G_{\beta,n}$ by \cite[Theorem 3.4]{SZ2}.  More precisely,
\begin{eqnarray*}
\phi_n \hbar^{(n-2)}(L_{1})\phi_n^{-1}&=&\hbar^{(n-2)}(\phi_2) \hbar^{(n-2)}(L_{1})\hbar^{(n-2)}(\phi_2^{-1})\\
&=& \hbar^{(n-2)}(\phi_2L_{1}\phi_2^{-1})~~(\text{by Lemma \ref{h morphism} })\\
&=& \hbar^{(n-2)}(L_1^{m_2} L_{\beta}^{n_2})~~(\text{by \ref{2-conj}})\\
&=& \Big(\hbar^{(n-2)}(L_1)\Big)^{m_2}\Big(\hbar^{(n-2)}(L_\beta)\Big)^{n_2},
\end{eqnarray*}
and
\begin{eqnarray*}
\phi_n \hbar^{(n-2)}(L_{\alpha})\phi_n^{-1}&=&\hbar^{(n-2)}(\phi_2) \hbar^{(n-2)}(L_{\alpha})\hbar^{(n-2)}(\phi_2^{-1})\\
&=& \hbar^{(n-2)}(\phi_2L_{\alpha}\phi_2^{-1})\\
&=& \hbar^{(n-2)}(L_1^{m_1} L_{\beta}^{n_1})\\
&=& \Big(\hbar^{(n-2)}(L_1)\Big)^{m_1}\Big(\hbar^{(n-2)}(L_\beta)\Big)^{n_1}.
\end{eqnarray*}
For $j=0,1,\cdots, n-3$,
\begin{eqnarray*}
\phi_n \hbar^{(j)}(L_{1})\phi_n^{-1}&=&\hbar^{(n-2)}(\phi_2) \hbar^{(j)}(L_{1})\hbar^{(n-2)}(\phi_2^{-1})\\
&=& \hbar^{(n-2-j)}\hbar^{(j)}(\phi_2L_{1}\phi_2^{-1})\hbar^{(n-2-j)}(\phi_2^{-1})\\
&=&  \hbar^{(j)} \Big(\hbar^{(n-2-j)}(\phi_2)L_1\hbar^{(n-2-j)}(\phi_2^{-1})\Big)\\
&=&  \hbar^{(j)}(L_1). ~~(\text{by Lemma \ref{h morphism} })
\end{eqnarray*}

Therefore, under the standard bases  $\big\{\hbar^{(n-2)}(L_1),\hbar^{(n-2)}(L_{\alpha}),\hbar^{(n-3)}(L_{1})$, $\cdots,\hbar^{(1)}(L_{1}), L_1\big\}$ and $\left\{\hbar^{(n-2)}(L_1),\hbar^{(n-2)}(L_{\beta}),\hbar^{(n-3)}(L_{1}), \cdots,\hbar^{(1)}(L_{1}), L_1\right\}$ of $G_{\alpha,n}$ and $G_{\beta,n}$ respectively, $\phi_n$ can be represented by matrix
\[
{\widetilde A_{\alpha,\beta}}=\left(
\begin{array}{cc}
A_{\alpha,\beta} & O\\
O & I
\end{array}
\right),
\]
where $A_{\alpha,\beta}=\left(
\begin{array}{cc}
m_2& m_1\\
n_2& n_1
\end{array}
\right) \in GL(2, \mathbb Z)$ and $I$ is the identity matrix of rank $n-2$. We call $\phi_n$ so defined the {\it standard conjugation} between $G_{\alpha,n}$ and $G_{\beta, n}$.

Now, we want to determine the matrix representation of the group $N_{\text{Homeo}_{+}(\mathbb{R})}(G_{\alpha,n})$ with respect to the standard
basis $\{e_1,...,e_n\}$.

If $\varphi\in\text{Homeo}_{+}(\mathbb{R})$ such that $\varphi G_{\alpha,2} \varphi^{-1}=G_{\alpha,2}$, then there are
integers $m_1,n_1,$ $m_2,n_2\in \mathbb{Z}$ with $|m_1n_2-n_1m_2|=1$ such that $L_1^{m_1}L_\alpha^{n_1}=\varphi L_\alpha\varphi^{-1}$,
$L_1^{m_2}L_\alpha^{n_2}=\varphi L_1\varphi^{-1}$, and $\alpha=\frac{m_1+n_1\alpha}{m_2+n_2\alpha}$ (See \cite{SZ2} Lemma 3.3). So, the matrix representation
of $\varphi$ belongs to the following group
 \[
F_\alpha:=\left\{\left(
\begin{array}{cc}
m_2& m_1\\
n_2& n_1
\end{array}
\right)\in GL(2,\mathbb Z):\alpha=\frac{m_1+n_1\alpha}{m_2+n_2\alpha}\right\}.
\]
 Thus by the definition of
$G_{\alpha,n}$, we see that each element in $N_{\text{Homeo}_{+}(\mathbb{R})}(G_{\alpha,n})$ has the matrix representation:
\[
{\widetilde B_\alpha}=\left(
\begin{array}{cc}
f_{\alpha} & *\\
O & B
\end{array}
\right)\in GL(n,\mathbb Z),
\]
where $B$ is an $(n-2)\times (n-2)$ upper triangular matrix  with diagonals $1$ and $f_\alpha\in F_\alpha$.
Here we remark that if $\alpha$ is not an algebraic number of degree $2$ over $\mathbb Q$, then $F_\alpha$ is trivial.
Conversely, given a matrix of the form $\widetilde B_\alpha$ as above, then there is a $\varphi\in N_{\text{Homeo}_{+}(\mathbb{R})}(G_{\alpha,n})$
whose matrix representation is $\widetilde B_\alpha$ by the construction process as in \cite{SZ2}.

Moreover, it is clear that
\begin{displaymath}
{\rm Conj}(G_{\alpha, n}, G_{\alpha', n})=\{\varphi\circ\psi_{\alpha,\alpha'}:\varphi\in N_{\text{Homeo}_{+}(\mathbb{R})}(G_{\alpha,n })\},
\end{displaymath}
where $\psi_{\alpha,\alpha'}\in {\rm Conj} (G_{\alpha, n}, G_{\alpha', n})$. So the matrix representation of each element in ${\rm Conj}(G_{\alpha, n}, G_{\alpha', n})$ has the form:
$\widetilde{B_\alpha}\widetilde{A_{\alpha,\alpha'}}$.

Altogether, we get a restatement of Theorem \ref{classification of S}.

\begin{theorem}
  $G_{\alpha,n,k,g,f}$ is topologically conjugate to $G_{\alpha',n',k',g',f'}$ if and only if
 \begin{itemize}
   \item $n=n'$ and $k=k'$;
   \item there exists
   \[
A_{\alpha,\alpha'}=\left(
\begin{array}{cc}
m_2& m_1\\
n_2& n_1
\end{array}
\right)\in GL(2,\mathbb Z),
\]
an upper triangular matrix $B\in GL(n-2, \mathbb Z)$ with diagonals $1$, and $\overrightarrow{w}\in k\mathbb Z^n$ such that $$\alpha'=\frac{m_1+n_1\alpha}{m_2+n_2\alpha},$$
and
   $$\overrightarrow{v}=\widetilde{B_\alpha}\widetilde{A_{\alpha,\alpha'}}\overrightarrow{u}+\overrightarrow{w},$$
where    $\overrightarrow{u}, \overrightarrow{v}$ are vectors in $\mathbb Z^n$ corresponding to $g$ and $g'$ respectively.

 \end{itemize}
\end{theorem}

\subsection*{Acknowledgements}
 We are grateful to the referee for his comments, corrections and helpful suggestions.

\end{document}